\documentclass[a4paper, 11pt]{article}
\usepackage{graphicx} 
\usepackage[margin=2.5cm]{geometry}
\usepackage{url}
\usepackage{xcolor}
\usepackage{amsmath,amssymb, amsthm}
\usepackage{authblk}
\usepackage{xspace}
\usepackage[T1]{fontenc}
\usepackage{hidde-standard-preamble}

\newcommand{\threelistcol}{\textsc{list}-$3$-\textsc{colouring}\xspace}
\newcommand{\yes}{\textsc{yes}}

\DeclareEmphSequence{\itshape\color{blue}}

\title{Faster 3-colouring algorithm for graphs of diameter 3}
\author[1]{Carla Groenland\thanks{Research supported by the Dutch Research Council (NWO, VI.Veni.232.073).}}
\author[2]{Hidde Koerts}
\author[2]{Sophie Spirkl\thanks{We acknowledge the support of the Natural Sciences and Engineering Research Council of Canada (NSERC), [funding reference number RGPIN-2020-03912].
    Cette recherche a \'et\'e financ\'ee par le Conseil de recherches en sciences naturelles et en g\'enie du Canada (CRSNG), [num\'ero de r\'ef\'erence RGPIN-2020-03912]. This project was funded in part by the Government of Ontario. This research was conducted while Spirkl was an Alfred P. Sloan Fellow.}}

\affil[1]{Delft Institute of Applied Mathematics, Technische Universiteit Delft, The Netherlands}
\affil[2]{Department of Combinatorics and Optimization, University of Waterloo, Canada}

\begin{document}
	\maketitle
    
    \begin{abstract}
        We show that given an $n$-vertex graph $G$ of diameter 3 we can decide if $G$ is $3$-colourable in time $2^{O(n^{2/3-\varepsilon})}$ for any $\varepsilon < 1/33$. This improves on the previous best algorithm of $2^{O((n\log n)^{2/3})}$ from Dębski, Piecyk and Rzążewski [Faster 3-coloring of small-diameter graphs, \textit{ESA 2021}].
    \end{abstract}

	\section{Introduction}

    In the past two decades, much attention has been given to when central NP-hard algorithmic problems on graphs can be solved faster when restrictions are placed on the input graph. For example, after a sequence of works~\cite{P7paper,4colouringp6free,subexp,Hoang10eaP5,Holyer81,Huang16,Leven83Galil,quasipoly}, 
    the complexity of \textsc{$k$-colouring} has been determined on $H$-free graphs for all connected $H$, with the remaining open case(s) being $k=3$ and $H=P_t$ a path on $t\geq 8$ vertices.    

    For \textsc{3-colouring}, a natural algorithmic approach (which is also the one used for the first subexponential time algorithm for $P_t$-free graphs) is to `branch' for well-chosen vertices on the colour that is assigned to them. This can be formalised by first generalising \textsc{3-colouring} to \textsc{list-3-colouring}, where lists $L(v)\subseteq \{1,2,3\}$ are given for each vertex $v$ of the input graph and we need to decide whether the graph admits a proper colouring that assigns each vertex a colour from its list. Often problems are restricted to hereditary classes (that is, those closed under taking induced subgraphs) which is convenient since `branching' on a colour keeps the graph within the class. This makes a recursive argument possible.

    An example of a natural restriction which does not lead to a hereditary graph class, is bounding the diameter of the input graph. This setting has been explored for many problems, such as problems related to various colouring variants~\cite{brause2022acyclic,piecyk2024c}, feedback vertex sets~\cite{feedbackvertexsets}, partitioning~\cite{partitioning}, and dominating sets~\cite{BouquetDelbotPicouleauDomSetBddDiam}. For $3$-colouring specifically, various authors have considered the setting of bounding the diameter, sometimes in combination with other constraints~\cite{campos2021coloring,DebskiPiecykRzazewski,MartinPaulusmaSmith3colclawfreegirthbdddiam, MARTIN2022150, martin2019colouring, MertziosSpirakis16}.
    Many problems remain NP-hard when adding a restriction on the diameter: for example, there is an easy reduction from \textsc{3-colouring} to \textsc{4-colouring} graphs of diameter~$2$ (by adding a universal vertex to the input graph). 
    In particular, assuming the Exponential Time Hypothesis (ETH), there is no subexponential-time algorithm for $k$-colouring diameter-$2$ graphs for any $k\geq 4$. Since \textsc{2-colouring} is polynomial-time solvable via reduction to $2$-\textsc{SAT}~\cite{2satforcolouring}, this reduction does not work for $k=3$.
    In fact, $3$-colouring can be solved in subexponential time on diameter-$2$ graphs and diameter-$3$ graphs, though not on diameter-4 graphs assuming ETH~\cite{DebskiPiecykRzazewski,MertziosSpirakis16}.

    In this paper, we study \textsc{3-colouring} on graphs of diameter~$3$. Mertzios and  Spirakis~\cite{MertziosSpirakis16} proved that under ETH, this problem cannot be solved in time $2^{o(\sqrt {n})}$ where $n$ denotes the number of vertices of the input graph, whereas Dębski, Piecyk and Rzążewski~\cite{DebskiPiecykRzazewski} provided an algorithm running in time $2^{O((n\log n)^{2/3})}$. This leaves the question of whether one of these two bounds may be close to tight. We show the upper bound is not tight, providing a slightly faster algorithm.

    \begin{theorem}\label{thm:main-theorem}
        \textsc{list-3-colouring} can be solved in time $2^{O(n^{2/3-\varepsilon})}$ for any fixed $\varepsilon < 1/33$ on the class of $n$-vertex graphs of diameter $3$.     
    \end{theorem} 
    The restriction of $\varepsilon < 1/33$ follows from a probabilistic argument used for proving the main structural result forming the basis of our algorithm.
    
    The complexity of \textsc{(list-)3-colouring} diameter-2 graphs also remains an interesting open problem and it is possible that this problem is solvable in polynomial time. 
    
    Mertzios and Spirakis~\cite{MertziosSpirakis16} provided the first subexponential-time algorithm for \textsc{list}-$3$-\textsc{colouring} on diameter-2 graphs with running time $2^{O(\sqrt{n \log n})}$. In the particular case of diameter-2 graphs, if we colour a vertex $v$ and all its neighbours, then this immediately leaves a \textsc{2-list-colouring} instance (which can be solved in polynomial time). This immediately provides an algorithm running in time $3^{
\delta(G)}n^{O(1)}$ where $\delta(G)$ denotes the minimum degree of the graph $G$. On the other hand, high degree vertices are also useful, since `branching' on their colour may remove many colours from the lists of their neighbours. We can then track the progress of our algorithm through a potential function such as $\sum_{v\in V(G)}|L(v)|$ and roughly speaking the recursive calculations will work out because at most one `branch' has a small improvement whereas all other branches have an improvement of at least $\Delta(G)/2$ where $\Delta(G)$ denotes the maximum degree of graph $G$. Instead of such a branching approach, Mertzios and Spirakis~\cite{MertziosSpirakis16}  find a small dominating set when $\delta(G)$ is large to obtain their running time.
    
    Similarly, for a diameter-3 graph we may aim for a running time of $2^{O((n\log n)^{2/3})}$ as follows. When the maximum degree is at most $(n\log n)^{1/3}$, then the second neighbourhood of any vertex is of size at most $(n\log n)^{2/3}$. Guessing the colours of the vertices in $v\cup N(v)\cup N(N(v))$ of a vertex $v$ in a diameter-3 graph leaves a \textsc{2-list-colouring} instance again which is solvable in polynomial time. By `branching' on the possible colours for a vertex of highest degree, removing vertices with a single colour in their list and updating colours throughout, we can hope to reduce to such a situation.
    
    The only small issue with this reasoning is that the graph classes of diameter-2 and diameter-3 graphs are not hereditary, and so we cannot simply `ignore' vertices whose colour has already been assigned while keeping our diameter assumption. Nevertheless, with some additional structural reasoning, the overall approach of branching on the colours of vertices of maximum degree until there is a vertex of small degree can be made to work. Dębski, Piecyk and Rzążewski~\cite{DebskiPiecykRzazewski} provide such an algorithm that solves \textsc{list-3-colouring} on $n$-vertex graphs of diameter~$3$ in time $2^{O((n\log n)^{2/3})}$.

    Our algorithm builds on this strategy, and is similar in approach to the algorithm Dębski, Piecyk and Rzążewski~\cite{DebskiPiecykRzazewski} introduce for the diameter-$2$ setting. To extend their approach, we use new structural insights as well as introducing additional branching rules to handle the issue of increasing diameter when removing coloured vertices.
    
	\section{Preliminaries}
    All graphs we consider in this paper are simple and finite. 
    We use $[k]$ to denote the set $\{1, \ldots, k\}$ for positive integers $k$ and $\log n$ for the natural ($e$-based) logarithm. 

    A problem instance for \textsc{list}-$3$-\textsc{colouring} consists of a graph $G$ and a list assignment $L : V(G) \to 2^{[3]}$ (that is, $L(v) \subseteq \{1,2,3\}$ for all $v \in V(G)$). The decision problem is then to determine whether there exists a proper colouring $f : V(G) \to[3]$ of $G$ such that $f(v) \in L(v)$ for all $v \in V(G)$. A problem instance for \textsc{2-list-colouring} consists of a graph $G$ and a list assignment $L$ such that $|L(v)|\leq 2$ for all $v\in V(G)$.
	
	We say two instances $(G, L)$ and $(G', L')$ are \emph{equivalent} if $(G, L)$ is a \textsc{yes}-instance if and only if $(G', L')$ is.
	
Given a graph $G$ and a list assignment $L: V(G) \to 2^{[3]}$ and  $i \in [3]$, let 
    \[
    L_i=\{v \in V(G) \text{ with }|L(v)| = i\}.
    \]
    
    For a positive integer $k$ and a vertex $v \in V(G)$, let $N^{(k)}(v)$ be the set of all vertices at distance exactly $k$ from $v$, and let $N^{(\leq k)}(v)$ be the set of all vertices at distance at most $k$ from $v$, excluding $v$ itself. Let $N^{(\leq k)}[v]$ be the closed neighbourhood equivalent, that is $N^{(\leq k)}[v] := N^{(\leq k)}(v) \cup \{v\}$. Additionally, let $N^{(k)}_{L_i}(v)$ and $N^{(\leq k)}_{L_i}(v)$ for $i \in [3]$ be the sets $N^{(k)}(v) \cap L_i$ and $N^{(\leq k)}(v) \cap L_i$ respectively. If $k=1$, we omit the superscript and simply use $N_{L_i}(v)$. Analogously we define $N^{(\leq k)}_{L_i}[v]$ and $N_{L_i}[v]$ for the closed neighbourhood equivalents. Finally, for a set $S \subseteq V(G)$, we define $N^{(\leq k)}(S)$ to be the set of all vertices in $V(G) \setminus S$ at distance at most $k$ from $S$, and $N^{(\leq k)}[S] := N^{(\leq k)}(S) \cup S$.

    We say that a probabilistic event $E$ happens \emph{with high probability} with respect to a variable $\ell$ if $\lim_{\ell \to \infty} \mathbb{P}[E] = 1$. In this paper, the relevant variable will always be the number of vertices in the graph. 

    We also use the following Chernoff bounds (see~\cite{Chernoffcite} for a proof).
    \begin{lemma}[Chernoff bound]
    \label{lem:Chernoff}
        Let $X\sim \text{Bin}(n,p)$ be the sum of $n$ independent random variables taking value $0$ with probability $1-p$ and value $1$ with probability $p$. Then for $\delta\in [0,1]$,
        \[
        \mathbb{P}(|X-np|\geq \delta np)\leq \exp(-\delta^2np/3).
        \]
        For $\delta>0$,
        \[
        \mathbb{P}(X\geq (1+\delta) np)\leq \exp(-\delta^2np/(2+\delta)).
        \]
    \end{lemma}
    In particular, we use multiple times that if $\mathbb{E}[X]=np\geq 3\log n$ for $n$ sufficiently large, then 
    \[
    \mathbb{P}(X\geq 2\mathbb{E}[X])\leq \mathbb{P}(|X-np|\geq np)\leq \exp(-np/3)\leq \exp(-\log n)=n^{-1}\to 0
    \]
    as $n\to \infty$. 
    One time, we will require a slightly stronger version of the Chernoff bound above:
    \begin{lemma}[Chernoff bound - upper bound variant]
    \label{lem:Chernoff-upper-bound-variant}
        Let $X\sim \text{Bin}(n,p)$ be the sum of $n$ independent random variables taking value $0$ with probability $1-p$ and value $1$ with probability $p$. Then for $\delta > 0$ and $b \geq n p$, 
        \[
        \mathbb{P}(X\geq (1+\delta) b)\leq \exp(-\delta^2b/(2+\delta)).
        \]
    \end{lemma}
    The above lemma follows from the observation that the standard proof of Chernoff tail bounds based on the moment generating function (see~\cite{Chernoffcite}) goes through with an upper bound on the expected value of $X$ rather than the expected value itself as well.

	\section{Reduction and branching rules}
	We first consider a number of instance features we may exploit to reduce the complexity of the problem instance. Note that if no vertex contains three colours in its list, the problem may be reduced to $2$-\textsc{SAT}. Hence, we have the following theorem.

    \begin{theorem}[Edwards~\cite{2satforcolouring}]\label{thm:list-2-colouring-poly-time}
        Let $G$ be a graph and let $L : V(G) \to 2^{[3]}$ be a list assignment such that $|L(v)| \leq 2$ for all $v \in V(G)$. Then it is decidable in polynomial time whether $G$ admits a proper $L$-colouring.
    \end{theorem}

    As a first step in simplifying our problem instances, we have the following reduction rules:
    \begin{itemize}
        \item[{\textbf{R1}}] If there exists a vertex $v \in V(G)$ with $L(v) = \emptyset$, the instance is infeasible. Return \textsc{no}.
        \item[{\textbf{R2}}] If $L_3 = \emptyset$, the instance may be solved in polynomial time using \cref{thm:list-2-colouring-poly-time}.
        \item[{\textbf{R3}}] If there exist adjacent vertices $u, v \in V(G)$ with $L(u) = \{c\}$ for some $c \in [3]$ and $c \in L(v)$, remove $c$ from $L(v)$.
        \item[{\textbf{R4}}] If $|L_3| < c$ for some constant $c$, exhaustively guess the colours assigned to the vertices in $L_3$, and solve the rest of the instance using rule R2.
    \end{itemize}
    The constant $c$ used in R4 will be determined later.

    Note that none of these rules affects the feasibility of the instance, and each may be checked and executed in polynomial time. Moreover, for R2-R4, the function $\sum_{v \in V(G)} |L(v)|$ strictly decreases by applying the reduction rule. Hence, each of these reduction rules is applied only a polynomial number of times for a given instance.

    In addition to these reduction rules, we introduce a number of branching rules. In contrast to the reduction rules, these branching rules do not directly simplify the problem instance. Rather, the branching rules each generate a number of new instances, each of which is simpler than the original instance. Solving these simpler instances in turn allows us to solve the original instance, enabling a recursive algorithm. Before introducing the first branching rules, we need the following lemma.

    \begin{lemma}\label{lem:colour-option-path-via-L2-vertices}
        Let $u \in V(G)$.  Then there exists a colour $c \in L(u)$ such that if we assign $c$ to $u$ and apply reduction rule R3 repeatedly, we obtain an instance $(G, L')$ with $|L'_3| \leq |L_3|-|N_{L_3}(N_{L_2}(N_{L_2}(u)))|/3$. 
    \end{lemma}
    \begin{proof}
        Let $x, y \in L_2$ and $v\in L_3$ such that $u - x - y - v$ is a path of length 3 in $G$. Let $c \in L(x) \cap L(y)$ and let $c'$ be the other colour in $L(x)$ (that is, $L(x) = \{c, c'\})$. If we assign the colour $c'$ to $u$, applying reduction rule R3 will result in vertex $x$ being assigned colour $c$. Hence, applying R3 again results in $L(y)$ losing a colour, determining the colour of $y$ as well. As a result, $L(v)$ will lose the colour assigned to $y$ and hence $v$ will be removed from $L_3$. In particular, for each $v\in N_{L_3}(N_{L_2}(N_{L_2}(u)))$, there is a colour $c$ that can be assigned to $v$ such that $v$ is removed from $L_3$ and so the statement follows from the pigeonhole principle.
    \end{proof} 

    Let $\mu := |L_3|$. We have the following branching rules:
    \begin{itemize}
        \item[{\textbf{B1}}] If there exists a vertex $v \in L_2 \cup L_3$ with $|N_{L_3}(v)| \geq \mu^{1/3 + \varepsilon}$, branch on the colour assigned to $v$. 
        \item[{\textbf{B2}}] If there exists a vertex $v \in L_3$ that has via neighbours in $L_2$ at least $\mu^{1/3 + \varepsilon}$ second neighbours in $L_3$, that is, \[|N_{L_3}(N_{L_2}(v))| \geq \mu^{1/3+\varepsilon},\] then branch on the colour assigned to $v$.
        \item[{\textbf{B3}}] If there exists a vertex $v \in L_3$ with \[|N_{L_3}(N_{L_2}(N_{L_2}(v)))| \geq \mu^{1/3 + \varepsilon},\] using \cref{lem:colour-option-path-via-L2-vertices} there exists a colour $c \in [3]$ such that if we assign $c$ to $v$, the lists of at least $\mu^{1/3 + \varepsilon}/3$ vertices in $N_{L_3}(N_{L_2}(N_{L_2}(v)))$ decrease in size after repeatedly applying R3. We branch on this colour for $v$: either we assign $v$ the colour $c$, or we remove $c$ from the list $L(v)$.
        \item[{\textbf{B4}}] If there exist two distinct vertices $u, v \in L_3$ with \[|N_{L_3}(N(u) \cap N(v))|\geq \mu^{1/3 + \varepsilon},\] we branch on the colours of $u$ and $v$. If $u$ and $v$ are adjacent, we create six branches, one for each of the possible colour assignments for $u$ and $v$ in which they are assigned different colours. If $u$ and $v$ are non-adjacent, we additionally consider a branch where we merge $u$ and $v$ into a single vertex but do not fix any colours. 
    \end{itemize}

    We establish the correctness of these branching rules in the appendices (see Lemma~\ref{lem:correctnessB1-B4}).

    We say an instance $(G,L)$ is \emph{reduced} if none of the above reduction or branching rules can be applied to reduce the complexity of the instance. Note that other than merging two vertices in branching rule B4, none of the branching rules remove any of the vertices. When a vertex $v$ is in $L_1$, it has effectively been coloured and no longer plays a role for the remainder of a reduced instance. We do not delete it to maintain diameter at most $3$.
    
    Branching rules B2 and B3 specifically are designed to establish and maintain the property that $G[L_3]$ is close to having diameter $3$ in the following sense: 

    \begin{lemma}\label{lem:near-diameter-3}
        Let $(G, L)$ be a reduced instance with $G$ a diameter-$3$ graph. Then for all vertices $u \in L_3$, it holds that $|N_{G[L_3]}^{(\leq 3)}(u)| \geq \mu - 4\mu^{2/3 + 2\varepsilon}$.
    \end{lemma}
    \begin{proof}
        Let $u \in L_3$. 
        For each $v\in L_3$, there is a path between $u$ and $v$ in $G$ of length at most 3 since $G$ has diameter $3$. We prove the following stronger statement: there are at most $4\mu^{2/3 + \varepsilon}$ vertices $v \in L_3$ such that there is a $u- v$ path of length at most 3 containing vertices in $V(G) \setminus L_3$. 
        
        By the reduction rules, no vertex in $L_1$ is adjacent to a vertex in $L_3$, and so a $u-v$ path of length 3 has no vertices in $L_1$. Therefore, we consider the paths of length at most $3$ containing vertices in $L_2$, which must have one of the following forms:
        \begin{enumerate}
            \item $u - x - v$ where $x \in L_2$,
            \item $u - x - y - v$ where $x, y \in L_2$,
            \item $u - x - y - v$ where $x \in L_2$ and $y \in L_3$, or
            \item $u - x - y - v$ where $x \in L_3$ and $y \in L_2$.
        \end{enumerate}
        Because B2 is not applicable to the reduced instance, we find that $|N_{L_3}(N_{L_2}(u))|<\mu^{1/3+\varepsilon}$ and in particular there are at most $\mu^{1/3 + \varepsilon}$ vertices $v \in L_3$ reachable from $u$ via paths of type 1.
        Similarly, because B3 is not applicable, there are at most $\mu^{1/3 + \varepsilon}$ vertices $v \in L_3$ reachable via paths of type 2. Because B1 and B2 are not applicable, we moreover find that there are at most $\mu^{2/3 + 2\varepsilon}$ vertices $v \in L_3$ reachable via paths of type 3, and at most $\mu^{2/3 + 2\varepsilon}$ vertices $v \in L_3$ reachable via paths of type 4.
        In total, there are at most $4\mu^{2/3 + 2\varepsilon}$ vertices $v \in L_3$ reachable from $u$ via paths of length at most $3$ that are not fully contained in $G[L_3]$, as desired.     
    \end{proof}
    
    In the next section, we will describe one more branching rule, which is only applied when the instance is reduced. 

    \section{Main lemma and the last branching rule}

    To introduce the final branching rule, we first require the following structural result that we will apply to $H=G[L_3]$. See \cref{fig:magic-lemma-outcomes} for an illustration of the second and third outcomes of the lemma below.

    \begin{figure}
        \centering
        \begin{subfigure}[b]{0.47\linewidth}
            \centering
            \includegraphics[width=\textwidth]{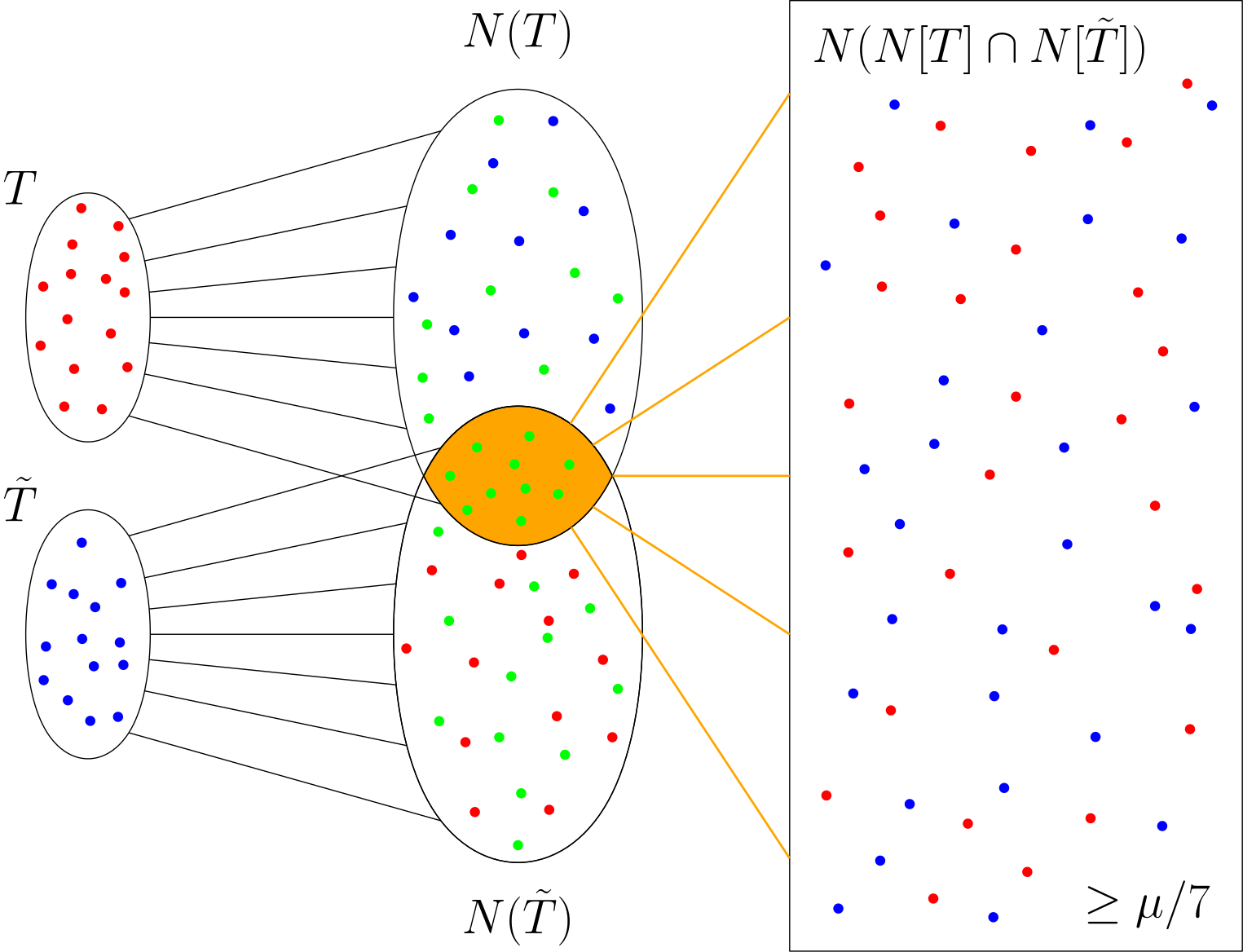}
            \caption{Outcome 2.}
            \label{fig:outcome2}
        \end{subfigure}
        \hfill
        \begin{subfigure}[b]{0.47\linewidth}
         \centering
         \includegraphics[width=\textwidth]{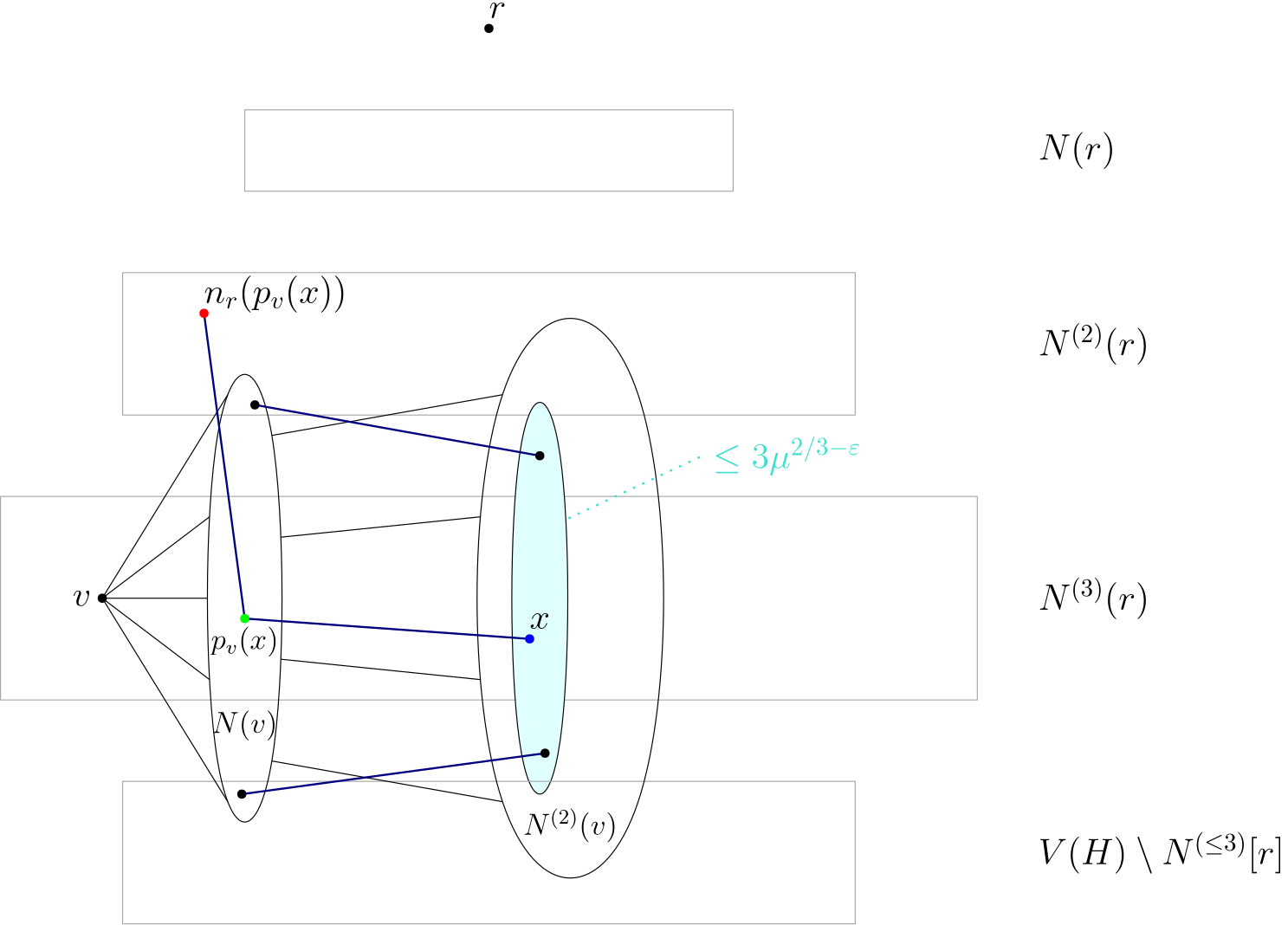}
         \caption{Outcome 3.}
         \label{fig:outcome3}
        \end{subfigure}
        \caption{Illustrations of the second and third outcomes in \cref{lem:magic-lemma} respectively.}
        \label{fig:magic-lemma-outcomes}
    \end{figure}
    \begin{lemma}\label{lem:magic-lemma}
        Let $\varepsilon < 1/33$. There exists an $\ell > 0$ such that the following holds for all $\mu \geq \ell$. 
        
        Let $H$ be a $\mu$-vertex graph such that for each vertex $v \in V(H)$,
        \begin{itemize}
            \item $|N(v)| \leq \mu^{1/3 + \varepsilon}$ and $|N^{(2)}(v)| \leq \mu^{2/3 + 2 \varepsilon}$, and
            \item $|N^{(\leq 3)}(v)| \geq \mu - 4\mu^{2/3+2\varepsilon},$
        \end{itemize}
        and for all pairs of distinct vertices $u, v \in V(H)$,
        \begin{itemize}
            \item $|N(N(u) \cap N(v))| \leq \mu^{1/3+ \varepsilon}$.
        \end{itemize} 
        Let $r \in V(H)$. For each vertex $a \in N^{(3)}(r)$, fix a neighbour $n_r(a) \in N^{(2)}(r)$. Additionally, for each $v \in V(H)$ and each $x \in N^{(2)}(v)$, fix a vertex $p_v(x) \in N(v) \cap N(x)$. Then one of the following statements must hold.
        \begin{enumerate}
            \item $H$ is not $3$-colourable.
            \item $H$ admits a $3$-colouring $\phi: V(H) \to [3]$ and there exist disjoint sets $T, \widetilde{T} \subseteq V(H)$ of size at most $2\mu^{2/3 - \varepsilon}$ such that \[|N(N[T] \cap N[\widetilde{T}])| \geq  \mu/7,\] both $T$ and $\widetilde{T}$ are monochromatic under $\phi$, and the vertices in $T$ receive a colour different than the one the vertices in $\widetilde{T}$ receive. 
            \item $H$ admits a $3$-colouring $\phi: V(H) \to [3]$ and there exists a vertex $v \in V(H)$ such that \[
            |\{x \in N^{(2)}(v):p_v(x) \not \in N^{(3)}(r) \text{ or }\phi(n_r(p_v(x)))\neq \phi(x)\}|\leq 3\mu^{2/3-\varepsilon}.
            \]
        \end{enumerate}
        Moreover, if $H$ admits a $3$-colouring $\phi: V(H) \to [3]$, then either outcome 2 or 3 must hold for this specific $3$-colouring $\phi$.
    \end{lemma}
    \begin{proof}
        We may assume that $H$ is $3$-colourable, as otherwise outcome 1 of the lemma holds.
        Let $\phi: V(H) \to [3]$ be a proper $3$-colouring of $G$. It remains to show that either outcome 2 or 3 holds for $\phi$.
        
        As in the assumptions of the lemma, let $r \in V(H)$ (arbitrarily). 
        We randomly pick two sets of vertices $S$ and $\widetilde{S}$ as follows.
        \begin{itemize}
            \item Each vertex in $V(H)$ is included in $S$ independently with probability $\mu^{-(1/3 + \varepsilon)}$.
            \item Each vertex in $N^{(2)}(r)$ is included in $\widetilde{S}$ independently with probability $\mu^{-3\varepsilon}$.
        \end{itemize}
        We show that with high probability, either there are (monochromatic) subsets $T$ and $\widetilde{T}$ of $S$ and $\widetilde{S}$ respectively such that the second outcome in the lemma applies, or the third outcome in the lemma applies.  Since $|V(H)| = \mu$ and $|N^{(2)}(r)| \leq \mu^{2/3+2\varepsilon}$, the following claim follows from Chernoff bounds (see \cref{lem:Chernoff}).
        \begin{claim}\label{claim:asymptotic-size-sets}
            With high probability, $|S| \leq 2\mu^{2/3 - \varepsilon}$ and $|\widetilde{S}| \leq 2\mu^{2/3 - \varepsilon}$.
        \end{claim}
        Let $R := V(H) \setminus N^{(3)}(r)$.
        Let $V' \subseteq N^{(3)}(r)$ be the set of all vertices $v \in N^{(3)}(r)$ such that $|R \cap N(v)| \leq \mu^{1/3 - 2\varepsilon}$.
        We show that most of the vertices in $H$ belong to $V'$.        
        \begin{claim}\label{claim:size-of-V'}
            For $\mu$ sufficiently large, $|V'| \geq \mu - 7\mu^{2/3 + 5\varepsilon}$.
        \end{claim}
        \begin{subproof}
            We first observe that by the first three assumptions in the lemma statement, 
\[                |N^{(\leq 2)}[r]| \leq \mu^{1/3 + \varepsilon} + \mu^{2/3 + 2\varepsilon}+1\text{ and } |V(H) \setminus N^{(\leq 3)}(r)| \leq 4\mu^{2/3 + 2 \varepsilon}.\]
This implies that $|R|=|V(H)\setminus N^{(3)}(r)| \leq 6\mu^{2/3 + 2\varepsilon}\leq \mu^{2/3+5\varepsilon}$ for $\mu$ sufficiently large.

            Let $B := N^{(3)}(r) \setminus V'$ be the set of all vertices $v \in N^{(3)}(r)$ such that $|R \cap N(v)| > \mu^{1/3 - 2\varepsilon}$. 
            We show that $|B| \leq 6\mu^{2/3+5\varepsilon}$.
            Let $E' \subseteq E(H)$ be the set of all edges with one endpoint in $B$ and one endpoint in $R$. By the definition of $B$, every vertex $v \in B$ is adjacent to more than $\mu^{1/3 - 2 \varepsilon}$ vertices in $R$. Therefore, $|E'| > \mu^{1/3 - 2\varepsilon} \cdot |B|$. 
            Since each vertex in $H$ has degree at most $\mu^{1/3+\varepsilon}$, we find that
            \[
 \mu^{1/3 - 2\varepsilon} \cdot |B|<|E'|\leq |R|\mu^{1/3+\varepsilon}\leq 6\mu^{2/3 + 2\varepsilon} \cdot \mu^{1/3 + \varepsilon}.
\]This shows that $|B| \leq  6\mu^{2/3 + 5\varepsilon}$.
Since $|V'|=|V(H)| - |R| - |B|$, the claim follows.
        \end{subproof}
        It is not quite enough to show that $V' \subseteq N(N[S] \cap N[\widetilde{S}])$ for the second outcome of the lemma: ideally, the vertices in $S$ and $\tilde{S}$ that are used to achieve this feat also have different colours. 
        For $i\in [3]$, let $S_i=S\cap \phi^{-1}\{i\}$ and $\widetilde{S}_i=\widetilde{S}\cap \phi^{-1}\{i\}$. We will show that with high probability, either the third outcome of the lemma holds, or for all $v\in V'$,
        \begin{equation}
        \label{eq:good_v}
            v\in \bigcup_{i\neq j} N(N(S_i)\cap N(\widetilde{S}_j)).  
        \end{equation}
        The second outcome then follows in the latter case from the previous claims and the pigeonhole principle.

        As in the assumptions of the lemma, for each vertex $a \in N^{(3)}(r)$, we fix a neighbour $n_r(a) \in N^{(2)}(r)$ and for all vertices $v \in V(H)$ and $x \in N^{(2)}(v)$, we fix a vertex $p_v(x)\in N(v)\cap  N(x)$. We call $p_v(x)$ the \emph{favourite parent} of $x$ with respect to $v$. 
        
        For a vertex $v \in V'$, we say that a vertex $x \in N^{(2)}(v)$ is \emph{$v$-deducing} if $p_v(x)\in N^{(3)}(r)$ and $\phi(x)\neq \phi(n_r(p_v(x))$. 
        Note that if moreover $x\in S$ and $n_r(p_v(x))\in \widetilde{S}$, then $p_v(x)\in N(S_i)\cap N(\widetilde{S}_j)$ for $i=\phi(x)$ and $j=\phi(n_r(p_v(x)))$ with $i\neq j$. Hence $v$ satisfies (\ref{eq:good_v}). 
               
        Recall that $n_r(a)$ is only defined for $a\in N^{(3)}(r)$. 
        For each vertex $v \in V'$ and each vertex $a \in N(v)$, we define the \emph{bucket} 
\[
B_a^{(v)} := \{x \in N^{(2)}(v) \, : \, p_v(x) = a\}\subseteq N(a).
\] 
        That is, the bucket $B_a^{(v)}$ is the set of all vertices in the second neighbourhood of $v$ which have $a$ as their favourite parent with respect to $v$.
        For a vertex $v \in V'$, we say that a vertex $a \in N(v)$ is \emph{$v$-fruitful} if $B_a^{(v)}$ contains at least $\mu^{1/3 - 2\varepsilon}$ $v$-deducing vertices. Note that this implies that $a\in N^{(3)}(r)$ by definition of $v$-deducing. Similarly, we call the corresponding bucket $B_a{(v)}$ itself $v$-fruitful as well. 

        \begin{claim}\label{claim:third-outcome}
            If there exists a vertex $v \in V'$ with at most $\mu^{1/3 - 2\varepsilon}$ $v$-fruitful buckets, then vertex $v$ satisfies the third outcome of the lemma.
        \end{claim}
        \begin{subproof}
        Suppose that $v\in V'$ has at most $\mu^{1/3 - 2\varepsilon}$ vertices $a \in N(v)\cap N^{(3)}(r)$ that are $v$-fruitful.
        We need to show that $|\{x\in N^{(2)}(v):p_v(x) \not\in N^{(3)}(r)$ or $\phi(n_r(p_v(x))) \neq \phi(x)\}|\leq 3  \mu^{2/3 - \varepsilon}$. 

        We first show that for each $v\in V'$, we have
            $|\{x\in N^{(2)}(v):p_v(x) \not \in N^{(3)}(r)\}|\leq \mu^{2/3-\varepsilon}$.
        Since $v\in V'$, by definition $|N(v)\setminus N^{(3)}(r)|\leq \mu^{1/3-2\varepsilon}$. For each vertex $y\in N(v)\setminus N^{(3)}(r)$, there are at most $|N(y)|\leq \mu^{1/3+\varepsilon}$ vertices $x$ with $p_v(x)=y$. Hence there are at most $\mu^{1/3-2\varepsilon}\mu^{1/3+\varepsilon}=\mu^{2/3-\varepsilon}$ vertices $x$ with $p_v(x) \not\in N^{(3)}(r)$. 

    It now suffices to show that $|\{x \in N^{(2)}(v):x\text{ is $v$-deducing}\}|\leq 2 \mu^{2/3 - \varepsilon}$. 
    Since each vertex in $H$ has degree at most $\mu^{1/3 + \varepsilon}$,
\begin{itemize}
    \item there are at most $\mu^{1/3 + \varepsilon}$ buckets $B_a^{(v)}$ with $a\in N(v)$, and
    \item $|B_a^{(v)}|\leq \mu^{1/3 + \varepsilon}$ (since $B_a^{(v)}\subseteq N(a)$). 
\end{itemize}
 Our assumption on $v$ is that at most $\mu^{1/3 - 2\varepsilon}$ buckets are $v$-fruitful. The $v$-fruitful buckets together thus contain at most $\mu^{1/3-2\varepsilon} \cdot \mu^{1/3 + \varepsilon} = \mu^{2/3-\varepsilon}$ $v$-deducing vertices. 
    Moreover, each bucket that is not $v$-fruitful contains at most $\mu^{1/3-2\varepsilon}$ $v$-deducing vertices, and so the buckets that are not $v$-fruitful also contain at most $\mu^{1/3-2\varepsilon} \cdot \mu^{1/3+\varepsilon} = \mu^{2/3 - \varepsilon}$ $v$-deducing vertices. Hence $N^{(2)}(v)$ contains at most $2\mu^{2/3 - \varepsilon}$ $v$-deducing vertices, as desired.
        \end{subproof}
        For each vertex $v \in V'$, let $A_{v} \subseteq N(v)\cap N^{(3)}(r)$ be the set of $v$-fruitful vertices $a$ for which $B_a^{(v)} \cap S$ contains at least one $v$-deducing vertex $x$. In particular, $A_v$ is a random variable that depends on $S$. 
        We will later analyse the probability that $n_r(p_v(x))=n_r(a) \in \widetilde{S}$.
        \begin{claim}\label{claim:size-of-Av}
            With high probability, for all vertices $v \in V'$, $|A_{v}| \geq \mu^{1/3 - 5 \varepsilon}/(4\log\mu)$. 
        \end{claim}
        \begin{subproof}
            We first show that, with high probability, each bucket $B_a^{(v)}$ contains at most $2\log \mu$ vertices from $S$ for each $a\in N(v)$.
            Recall that $n:=|B_a^{(v)}| \leq |N(a)| \leq \mu^{1/3 + \varepsilon}$. Each element of $B_a^{(v)}$ is included in $S$ independently with probability $p:=\mu^{-(1/3+\varepsilon)}$. Let $X:=|B_a^{(v)}\cap S|$. Then $\mathbb{E}[X]=np\leq 1$. So by the variant of Chernoff bounds with an upper bound on the expected value (\cref{lem:Chernoff-upper-bound-variant} with $\delta=2\log \mu -1$ and $b=1$), and as for $\mu$ sufficiently large we have that $\delta^2/(2+\delta)\geq 0.9(\delta+1)$,
            \[
            \mathbb{P}(|B_a^{(v)}\cap S| > 2\log\mu )= \mathbb{P}(|B_a^{(v)}\cap S| >1+ \delta) \leq \exp(-0.9(\delta+1))= \mu^{-1.8}.
            \]
            There are at most $|N(v)|\leq \mu^{1/3+\varepsilon}$ buckets. So by the union bound, the probability that a bucket $B_a^{(v)}$ (for some $v\in V'$ and $a\in N(v)$) contains more than $2\log \mu $ vertices from $S$ is at most
            $\mu\cdot \mu^{1/3+\varepsilon}\cdot \mu^{-1.8}$, which tends to $0$ as $\mu \to \infty$.

        By \cref{claim:third-outcome}, we may assume that for all vertices $v \in V'$, there are at least $\mu^{1/3 - 2\varepsilon}$ $v$-fruitful vertices $a \in N(v)\cap N^{(3)}(r)$, meaning that that the bucket $B_a^{(v)}$ contains at least $\mu^{1/3-2\varepsilon}$ $v$-deducing vertices. Since these buckets are by definition disjoint, there exist at least $\mu^{2/3 - 4 \varepsilon}$ $v$-deducing vertices $x \in N^{(2)}(v)$ such that $p_v(x)$ is $v$-fruitful. Let $X_v$ denote the number of such vertices $x$ that are part of $S$. Each such $x$ is part of $S$ independently with probability $\mu^{-(1/3+\varepsilon)}$ and so $\mathbb{E}[X_v]\geq \mu^{1/3-5\varepsilon}$.  Hence, by Chernoff bounds (the first inequality of \cref{lem:Chernoff} with $\delta=1/2$), $\mathbb{P}(X_v\leq \mu^{1/3-5\varepsilon}/2)\leq \mu^{-2}$ for $\mu$ sufficiently large, using that $1/3-5\varepsilon>0$. So with high probability, for each $v\in V'$, the pigeonhole principle implies that there are at least $\frac{1}{2}\mu^{1/3 - 5\varepsilon} / (2\log \mu)$ different buckets that contain $v$-deducing vertices in $S$. Hence, with high probability, $|A_v| \geq \mu^{1/3 - 5\varepsilon}/(4\log \mu)$.
        \end{subproof}
        If for a vertex $v \in V'$ one of the vertices in $\{n_r(a) \, : \, a \in A_v\}$ is included in $\widetilde{S}$, then $v$ satisfies (\ref{eq:good_v}) as desired. We therefore next show that the set $\{n_r(a) \, : \, a \in A_v\}$ is sufficiently large.
        \begin{claim}\label{claim:size-of-nr-set}
            With high probability, for all vertices $v \in V'$, we have $|\{n_r(a) \, : \, a \in A_v\}| \geq 4\mu^{3\varepsilon}\log \mu$.
        \end{claim}
        \begin{subproof}            
            We show that $|\{n_r(a) \, : \, a \in A_v\}| \geq 20\mu^{3\varepsilon}\log \mu$ for any vertex $v$ with $|A_v| \geq \mu^{1/3 - 5\varepsilon}/(4\log  \mu)$. The claim then follows from \cref{claim:size-of-Av}.
            
            Suppose towards a contradiction that there is a vertex $v$ with $|A_v| \geq \mu^{1/3 - 5\varepsilon}/(4\log  \mu)$ yet $|\{n_r(a) \, : \, a \in A_v\}| < 4\mu^{3 \varepsilon}\log  \mu$.
            For $y\in N^{(2)}(r)$, let $A_v(y):= \{a \in A_v \, : \, n_r(a) = y\}$.
            By the pigeonhole principle, there exist a vertex $y \in N^{(2)}(r)$ such that $|A_v(y)|> \mu^{1/3-5\varepsilon} / ( \mu^{3\varepsilon}(4\log \mu)^2) = \mu^{1/3 -8 \varepsilon} / (4\log  \mu)^2$.  

            By definition of $A_v$, we have that each $a \in A_v(y)$ is $v$-fruitful, and hence $|B_a^{(v)}| \geq \mu^{1/3 - 2\varepsilon}$ for each $a \in A_v(y)$. Since all buckets are disjoint,            
            \begin{align*}
                \left|N(A_v(y))\right| &\geq \left| \bigcup_{a \in A_v(y)} B_{a}^{(v)}\right|\\
                &= \sum_{a \in A_v(y)} \left| B_a^{(v)} \right|\\
                &> \mu^{1/3 - 2\varepsilon} \cdot \mu^{1/3 - 8\varepsilon} /(4\log  \mu)^2\\
                &= \mu^{2/3 - 10 \varepsilon}/(4\log  \mu)^2.
            \end{align*}
            Each vertex in $A_v(y)$ is a neighbour of both $y$ and $v$ and so
              \[\left|N(N(y) \cap N(v))\right| \geq |N(A_v(y))| >  \mu^{2/3 - 10 \varepsilon}/(4\log  \mu)^2.\]
            Note that
            \[
            2/3-10\varepsilon>  1/3+\varepsilon \iff 1/3>11\varepsilon\iff \varepsilon<1/33.
            \]
            For $\mu$ sufficiently large, as $\varepsilon<1/33$, we found above that $|N(N(y)\cap N(v))|> \mu^{1/3+\varepsilon}$. 
            However, by the fourth assumption in the lemma statement, $|N(N(y)\cap N(v))|\leq \mu^{1/3+\varepsilon}$. This gives a contradiction.
        \end{subproof}
        
        Let $v\in V'$. Then by the previous claim we may assume that
        $|\{n_r(a) \, : \, a \in A_v\}| \geq 20\mu^{3\varepsilon}\log \mu$.
        Since each vertex in $N^{(2)}(r)$ is included in $\widetilde{S}$ independently with probability $\mu^{-3\varepsilon}$, by Chernoff bounds (\cref{lem:Chernoff} with $X=|\widetilde{S}\cap  \{n_r(a) \, : \, a \in A_v\}|$, $np\geq 20\log \mu$ and $\delta=1/2$), for $\mu$ sufficiently large
        \[
        \mathbb{P}(|\widetilde{S}\cap  \{n_r(a) \, : \, a \in A_v\}| < 10\log \mu)\leq \exp(-20\log \mu/12)\leq  \mu^{-1.1}.
        \]
        In particular, by a union bound, with high probability for each $v\in V'$ we find that $\widetilde{S}\cap  \{n_r(a) \, : \, a \in A_v\}\neq \emptyset$.

        We next spell out the final pigeonhole principle application.
        For each vertex $v\in V'$, we assign a pair of colours $(i,j)\in [3]^2$ with $i\neq j$ as follows. Let $a\in A_v$ with $n_r(a)\in \widetilde{S}$ and $x\in B_a^{(v)}\cap S$ be a $v$-deducing vertex. Let $i=\phi(x)$ and $j=\phi(n_r(a))=\phi(n_r(p_v(x)))$. By the definition of $v$-deducing, $i\neq j$.

        By the pigeonhole principle, there exists a pair $(i,j)$ which is assigned to $v$ for at least $\frac16|V'|$ vertices $v\in V'$. Let
        \begin{align*}
            T &:= \{ v \in S \,  :\, \phi(v) = i\},\\
            \widetilde{T} &:= \{v \in \widetilde{S} \, : \, \phi(v) = j\}.
        \end{align*}
        By \cref{claim:asymptotic-size-sets}, as $T \subseteq S$ and $\widetilde{T} \subseteq \widetilde{S}$, we have that $|T|, |\widetilde{T}| \leq 2 \mu^{2/3 - \varepsilon}$. By definition, $T$ and $\widetilde{T}$ are monochromatic under $\phi$ for different colours and \[
        |N(N[T] \cap N[\widetilde{T}])| \geq \frac16 |V'| \geq \frac16 (\mu -9\mu^{2/3+5\epsilon})\geq \frac17 \mu
        \] by Claim~\ref{claim:size-of-V'} for $\mu$ sufficiently large. This means the second outcome of the lemma holds.
        We define $\ell$ such that all `for all $\mu$ sufficiently large' statements hold. Note that $\ell$ only depends on our choice of $\varepsilon$.
    \end{proof}

    The constant $\ell$ given in the statement of \cref{lem:magic-lemma} is the value we will use for constant $c$ in reduction rule R4. 

    Based on \cref{lem:magic-lemma}, we introduce one final branching rule. 
    This branching rule is only applied if none of the previous reduction and branching rules are applicable. Namely, if the instance is reduced, $H=G[L_3]$ satisfies the conditions of \cref{lem:magic-lemma} (as shown in the proof of \cref{lem:correctnessB5}).

    \begin{itemize}
        \item[{\textbf{B5}}] If $(G, L)$ is reduced, we create a number of new instances. 
        
        As in the statement of \cref{lem:magic-lemma}, fix $r \in V(H)$,  neighbours $n_r(a) \in N_{L_3}^{(2)}(r)$ for vertices $a \in N_{L_3}^{(3)}(r)$, and vertices $p_v(x) \in N_{L_3}(v) \cap N_{L_3}(x)$ for $v \in L_3$ and $x \in N_{L_3}^{(2)}(v)$. 

        First, for each pair of disjoint subsets $T, \widetilde{T} \subseteq L_3$ such that each has size at most $2 \mu^{2/3-\varepsilon}$ and $|N_{L_3}(N_{L_3}[T] \cap N_{L_3}[\widetilde{T}])| \geq  \mu/7$, we create six new branches: one for each pair of colours $(i, j) \in [3]^2$ such that $i \neq j$. In each of the six respective branches, assign each vertex in $T$ the colour $i$, and each vertex in $\widetilde{T}$ the colour $j$. These branches correspond to the second outcome of \cref{lem:magic-lemma}.

        Second, we create branches corresponding to the third outcome of \cref{lem:magic-lemma}. For each vertex $v \in N_{L_3}^{(3)}(r)$, we create a branch for each possible valid combination of a 3-colouring of $N_{L_3}[v]$, a 3-colouring of $\{ n_r(a)\, : \, a \in N_{L_3}(v)\}$, a set $S \subseteq N_{L_3}^{(2)}(v)$ of size at most $3\mu^{2/3 - \varepsilon}$ such that $\{x \in N_{L_3}^{(2)}(v) \, : \, p_v(x) \not \in L_3\setminus N^{(3)}(r)\} \subseteq S$, and each 3-colouring of $S$. In these branches, colour the vertices of \[
        N_{L_3}[v]\cup \{ n_r(a)\, : \, a \in N_{L_3}(v)\}\cup S
        \]
        according to the colourings defining the branch. This gives a partial colouring $\phi$. We extend $\phi$ to the remaining vertices $x \in N_{L_3}^{(2)}(v) \setminus S$ via $\phi(x):= \phi(n_r(p_v(x)))$.
    \end{itemize}
    We establish the correctness of this final branching rule in the appendix (see Lemma~\ref{lem:correctnessB5}). 

    While branching rule B5 creates many new branches every time it is applied, its main purpose is to reduce the size of $L_3$. It indeed does so significantly on every application, as shown in the following lemma.

    \begin{lemma}\label{lem:size-of-L3-after-applying-B5}
        Let $(G,L)$ be a reduced instance and let $\mu=|L_3|$. Then, in the new instances created by applying branching rule B5, after exhaustively applying the reduction rules, the number of vertices of list size $3$ is 
        \begin{enumerate}
            \item at most $(6/7)\cdot \mu$ in the branches corresponding to the second outcome of \cref{lem:magic-lemma}, and
            \item at most $4\mu^{2/3 + 2\varepsilon}$ in the branches corresponding to the third outcome of \cref{lem:magic-lemma}.
        \end{enumerate}
    \end{lemma}
    \begin{proof}
        For the branches corresponding to the second outcome of \cref{lem:magic-lemma}, since $T$ and $\widetilde{T}$ are assigned different colours, the colour of each vertex in $N_{L_3}[T] \cap N_{L_3}[\widetilde{T}]$ is determined after applying the reduction rules exhaustively. Hence, each vertex in $N_{L_3}(N_{L_3}[T] \cap N_{L_3}[\widetilde{T}])$ loses a colour from its list. By the definition of branching rule B5 this set has size at least $\mu / 7$, and thus the created new instances each have at least $\mu/7$ fewer vertices of list size $3$, as desired.

        Next we consider the branches corresponding to the third outcome of \cref{lem:magic-lemma}. In each of these branches, for some vertex $v \in N_{L_3}^{(3)}$ we colour all of $N_{L_3}^{(\leq 2)}[v]$. Thus, after applying the reduction rules exhaustively, none of the vertices in $N_{L_3}^{(\leq 3)}[v]$ still have list size $3$. Hence, using \cref{lem:near-diameter-3}, the resulting instance has at most $4\mu^{2/3 + 2\varepsilon}$ vertices of list size $3$, as desired.
    \end{proof}

    \section{Running time analysis}

    Note that branching rule B5 always applies if the instance is in a reduced state. Since an instance is by definition reduced if and only if none of the other branching and reduction rules may be applied, the branching algorithm is guaranteed to terminate. Hence, by the correctness of the reduction and branching rules, the algorithm correctly determines whether the input graph is $3$-colourable and constructs a valid $3$-colouring if it exists. The fifth branching rule depends on the choice of $\varepsilon$. We show that for each $\delta\in (0,1/33)$, there exists an $\varepsilon\in (0,1/33)$ such that the branching algorithm using this $\varepsilon$ will run in time $2^{O(n^{2/3-\delta})}$ in the appendix (see Lemma~\ref{lem:running-time}). This proves \cref{thm:main-theorem}.
    
    We briefly sketch the runtime analysis here. We create a recurrence relation for each of the branching rules capturing the number of new instances it creates. For each of the branching rules B1-B4, at most one instance is created with one fewer vertex of list-size $3$, and a constant number (say $k$) of instances where the number of vertices in $L_3$ drops by at least $\mu^{1/3 + \varepsilon}$. If we look at the number of leaves in the corresponding tree (that is, the total number of instances created following rules B1-B4), we can describe each leaf by the path from the root. This will be a sequence with symbols in $\{0,1,\dots,k\}$, say, of length at most $n$ in this case (an upper bound for $\mu$), and at most $n^{2/3 - \varepsilon}$ of the entries are nonzero (``the big steps,'' whereas the label $0$ is used for branches along which the size decreases by $1$). This gives ${n \choose n^{2/3 - \varepsilon}} \cdot k^{n^{2/3 - \varepsilon}}$ as upper bound. Moreover, whereas branching rule B5 generates many more new instances, each of them drops the number of vertices in $L_3$ by either a constant factor or an exponential factor, which sufficiently counteracts the exponential explosion in the number of instances.

\section{Conclusion}
We made progress on one of the two `notorious'~\cite{Paulusma} open cases of the complexity of $k$-colouring on graphs of bounded diameter, by providing a faster algorithm for $3$-colouring diameter-3 graphs. Assuming the Exponential Time Hypothesis (ETH), 3-colouring cannot be solved in time $2^{o(n)}$ on diameter-4 graphs with $n$ vertices~\cite{MertziosSpirakis16}, nor can $k$-colouring be solved in time $2^{o(n)}$ on diameter-2 graphs for $k\geq 4$. The most intriguing open problem is whether diameter-2 graphs can be 3-coloured in polynomial time. We therefore explicitly highlight the following problem.
\begin{problem}
    Can \textsc{3-colouring} be solved on diameter-2 graphs in quasi-polynomial time? 
\end{problem}
We provided, for any $\varepsilon<1/33$, an algorithm for 
\textsc{3-colouring} on diameter-3 graphs that runs in time $2^{O(n^{2/3-\varepsilon})}$. The limiting factor on $1/33$ in our proof occurs in the proof of \cref{claim:size-of-nr-set}. Before our work, the exponent of $2/3$ could have been a natural guess for the optimal exponent, with the ETH-based lower bound on the exponent of $1/2$ by  Mertzios and Spirakis~\cite{MertziosSpirakis16} being another natural option.

\printbibliography

\section{Appendices}

\subsection{Correctness of branching rules B1-B4}
    \vspace{7mm}
        \begin{lemma}\label{lem:correctnessB1-B4}
        Let $(G, L)$ be a \threelistcol instance. For each of the branching rules B1-B4, at least one of the branches resulting from applying the branching rule is a \yes-instance if and only if $(G, L)$ is a \yes-instance.
    \end{lemma}
    \begin{proof}
        We proceed by case distinction on which branching rule is applied.
        \begin{description}
            \item[Cases B1-B3:] Follow from a proper $L$-colouring having to assign some colour in $L(v)$ to $v$.
            \item[Case B4:] If one of the branches where $u$ and $v$ are assigned different colours is a \yes-instance, the correctness is immediate. It remains to consider the case where $u$ and $v$ are non-adjacent and they are merged. Let $G'$ be the graph resulting from merging $u$ and $v$ in $G$, let $x$ be the vertex resulting from merging $u$ and $v$, and let $L': V(G') \to 2^{[3]}$ be the corresponding list assignment. If $(G', L')$ is a \yes-instance, let $\phi ': V(G') \to [3]$ be an $L'$-colouring of $G'$. By setting $\phi(y)=\phi'(y)$ for $y\in V(G)\setminus \{u,v\}=V(G')\setminus \{x\}$, $\phi(u):=\phi(x)$ and $\phi(v) := \phi(x)$, we obtain a valid $L$-colouring of $G$, as $u$ and $v$ are non-adjacent. 
            
            On the other hand, if $(G, L)$ is a \yes-instance, let $\phi$ be a valid $L$-colouring of $G$. If $\phi(u) = \phi(v)$, then the branch resulting from merging $u$ and $v$ will be a \yes-instance, witnessed by setting $\phi(x)$ for $x$ the vertex resulting from merging $u$ and $v$ to equal $\phi(u) = \phi(v)$. For the case where $\phi(u) \neq \phi(v)$ it follows directly that there is a branch where $u$ is assigned $\phi(u)$ and $v$ is assigned $\phi(v)$ which must be a \yes-instance.
        \end{description}
    \end{proof}

\subsection{Correctness of branching rule B5}
    \vspace{7mm}

\begin{lemma}\label{lem:correctnessB5}
        Let $(G,L)$ be a \threelistcol instance. For branching rule B5, at least one of the branches created by applying the branching rule is a \yes-instance if and only if $(G, L)$ is a \yes-instance.
    \end{lemma}    \begin{proof}
        We first claim that if an instance is reduced, $G[L_3]$ satisfies the conditions of \cref{lem:magic-lemma}. Firstly, by branching rule B1, we have $|N_{L_3}(v)| < \mu^{1/3 + \varepsilon}$ for all $v \in L_3$. Applying this inequality to all vertices in $N_{L_3}(v)$ gives us that 
        \[\left|N_{G[L_3]}^{(2)}(v)\right| \leq \sum_{u \in N_{L_3}(v)} |N_{L_3}(u)| < \mu^{1/3 + \varepsilon} \cdot \mu^{1/3 + \varepsilon} =  \mu^{2/3 + 2\varepsilon}.\]
        Moreover, by \cref{lem:near-diameter-3}, $|N_{G[L_3]}^{(\leq 3)}(v)| \geq \mu - 4 \mu^{2/3 + 2\varepsilon}$. 
        
        The final condition of \cref{lem:magic-lemma}, that for all distinct $u,v \in L_3$ we have $|N_{L_3}(N_{L_3} (u) \cap N_{L_3}(v))| \leq \mu^{1/3 + \varepsilon}$, follows from branching rule B4.
    
        Since a valid colouring in any of the created branches would be a valid colouring for the original instance $(G, L)$, it suffices to show that if $(G,L)$ is a \yes-instance, then so is one of the branches. Suppose that $\phi :V(G) \to [3]$ is a valid $L$-colouring of $G$. By \cref{lem:magic-lemma}, $\phi$ restricted to $G[L_3]$ admits structure according to either the second or the third outcome of \cref{lem:magic-lemma}. If the second outcome holds, we note that correctness directly follows from the first set of branches created by B5. 

        Hence, we focus on the third outcome. We claim that this outcome is covered by the second set of branches constructed by B5. Namely, for $v\in L_3$ the special vertex designated in the third outcome of \cref{lem:magic-lemma}, we consider the branch where $N_{L_3}[v]$ and $\{ n_r(a)\, : \, a \in N_{L_3}(v)\}$ are coloured according to $\phi$, and where $S$ is the set of all vertices $x \in N_{L_3}^{(2)}$ such that $p_v(x) \in R$ or $\phi(n_r(p_v(x))) \neq \phi(x)$. By \cref{lem:magic-lemma}, $\phi$ is a valid colouring for this branch, and hence it is a \yes-instance, as desired.
    \end{proof}

\subsection{Detailed running time calculation}
\vspace{7mm}

    \begin{lemma}\label{lem:running-time}
        For every $0<\varepsilon <1/33$ there exists a choice of parameters for the branching algorithm such that it terminates in $2^{O(n^{2/3-\varepsilon})}$.
    \end{lemma}
\begin{proof}
        Note that the parameters are encoded in the reduction and branching rules via $\varepsilon$. Let $F(x)$ be the maximum number of instances created by the algorithm in solving an instance $(G,L)$ with $\mu \leq x$. We will bound the running time by considering which branching rule was applied and using induction on $\mu$. For the reader's convenience, we will restate each branching rule as they come up in the analysis.

        \begin{itemize}
            \item[{\textbf{B1}}]\itshape  If there exists a vertex $v \in L_2 \cup L_3$ with $|N_{L_3}(v)| \geq \mu^{1/3 + \varepsilon}$, branch on the colour assigned to $v$. 
        \end{itemize}

        \begin{claim}
            If branching rule B1 was applied, then the number of instances is at most \[3F(\mu - \mu^{1/3 + \varepsilon}) + p(n).\]
        \end{claim}
        \begin{subproof}
            Let $v \in L_2 \cup L_3$ be the vertex such that we branch on the colour assigned to $v$. Note that we create three new instances (one for each colour assigned to $v$), and since $|N_{L_3}(v)| \geq \mu^{1/3 + \varepsilon}$, each of these branches contains at least $\mu^{1/3 + \varepsilon}$ fewer vertices with list size $3$.
        \end{subproof}

        \begin{itemize}
            \item[{\textbf{B2}}] \itshape If there exists a vertex $v \in L_3$ that has via neighbours in $L_2$ at least $\mu^{1/3 + \varepsilon}$ second neighbours in $L_3$, that is, \[|N_{L_3}(N_{L_2}(v))| \geq \mu^{1/3+\varepsilon},\] then branch on the colour assigned to $v$.
        \end{itemize}
        
        \begin{claim}
            If branching rule B2 was applied, the number of instances is at most \[F(\mu - 1) + 2 F(\mu - \mu^{1/3 + \varepsilon}/3) + p(n).\]
        \end{claim}
        \begin{subproof}
            Let $v \in L_3$ be the vertex such that we branch on the colour assigned to $v$. We again create three new instances, one for each colour assigned to $v$. Let $S := N_{L_3}(N_{L_2}(v))$. By the definition of B2, we have $|S| \geq \mu^{1/3 + \varepsilon}$. For every vertex $u \in S$, assign a favourite neighbour $u' \in N_{L_2}(v)$. Because each vertex in $L_2$ has list $\{1, 2\}$, $\{1, 3\}$, or $\{2, 3\}$, there exists a set $S' \subseteq S$ of size at least $|S|/3 \geq \mu^{1/3 + \varepsilon}/3$ such that the lists assigned to the favourite neighbours of all vertices in $S'$ are identical. Without loss of generality, assume that all these favourite neighbours have list $\{1, 2\}$. Then in the branches where $v$ is assigned colour $1$ or $2$, each vertex in $S'$ loses a colour from its list after exhaustively applying the reduction rules. Thus, in those two branches, the newly created instances have at least $\mu^{1/3 + \varepsilon}/3$ fewer vertices of list size $3$. Moreover, in the final branch, vertex $v$ itself no longer has list size $3$. 
        \end{subproof}

        \begin{itemize}
            \item[{\textbf{B3}}] \itshape If there exists a vertex $v \in L_3$ with \[|N_{L_3}(N_{L_2}(N_{L_2}(v)))| \geq \mu^{1/3 + \varepsilon},\] using \cref{lem:colour-option-path-via-L2-vertices} there exists a colour $c \in [3]$ such that if we assign $c$ to $v$, the lists of at least $\mu^{1/3 + \varepsilon}/3$ vertices in $N_{L_3}(N_{L_2}(N_{L_2}(v)))$ decrease in size after repeatedly applying R3. We branch on this colour for $v$: either we assign $v$ the colour $c$, or we remove $c$ from the list $L(v)$.
        \end{itemize}

        \begin{claim}
            If branching rule B3 was applied, the number of instances is at most \[F(\mu - 1) + F(\mu - \mu^{1/3 + \varepsilon}/3) + p(n).\]
        \end{claim}
        \begin{subproof}
            When applying this branching rule on $v\in L_3$, we either remove the colour $c$ from the list of $v$ (decreasing $|L_3|$ by one) or assign the colour $c$ to $v$ (decreasing $|L_3|$ by $\mu^{1/3+\varepsilon}/3$ by choice of $v$ and $c$ in B3).
        \end{subproof}

        \begin{itemize}
            \item[{\textbf{B4}}] \itshape If there exist two distinct vertices $u, v \in L_3$ with \[|N_{L_3}(N(u) \cap N(v))|\geq \mu^{1/3 + \varepsilon},\] we branch on the colours of $u$ and $v$. If $u$ and $v$ are adjacent, we create six branches, one for each of the possible colour assignments for $u$ and $v$ in which they are assigned different colours. If $u$ and $v$ are non-adjacent, we additionally consider a branch where we merge $u$ and $v$ into a single vertex but do not fix any colours.  
        \end{itemize}

        \begin{claim}
            If branching rule B4 was applied, the number of instances is at most \[F(\mu-1) + 6F(\mu - \mu^{1/3 + \varepsilon}) + p(n).\]
        \end{claim}
        \begin{subproof}
            Let $u, v \in L_3$ be the two vertices on whose colours we branch. Let $S := N_{L_3}(N(u) \cap N(v))$. By the definition of B4, we have that $|S| \geq \mu^{1/3 + \varepsilon}$. In each of the six branches where $u$ and $v$ are assigned different colours, the list of every vertex in $N(u) \cap N(v)$ has size at most $1$ after applying the reduction rules. Hence, each vertex in $S$ loses at least one colour, and therefore the number of vertices of list size $3$ decreases by at least $|S| \geq \mu^{1/3 + \varepsilon}$. Moreover, in the case that $u$ and $v$ are non-adjacent, by merging them, the number of vertices of list size $3$ decreases by one.
        \end{subproof}

        \begin{itemize}
            \item[{\textbf{B5}}] \itshape If $(G, L)$ is reduced, we create a number of new instances. 
        
            As in the statement of \cref{lem:magic-lemma}, fix $r \in V(H)$,  neighbours $n_r(a) \in N_{L_3}^{(2)}(r)$ for vertices $a \in N_{L_3}^{(3)}(r)$, and vertices $p_v(x) \in N_{L_3}(v) \cap N_{L_3}(x)$ for $v \in L_3$ and $x \in N_{L_3}^{(2)}(v)$. 

            First, for each pair of disjoint subsets $T, \widetilde{T} \subseteq L_3$ such that each has size at most $2 \mu^{2/3-\varepsilon}$ and $|N_{L_3}(N_{L_3}[T] \cap N_{L_3}[\widetilde{T}])| \geq  \mu/7$, we create six new branches: one for each pair of colours $(i, j) \in [3]^2$ such that $i \neq j$. In each of the six respective branches, assign each vertex in $T$ the colour $i$, and each vertex in $\widetilde{T}$ the colour $j$. These branches correspond to the second outcome of \cref{lem:magic-lemma}.

            Second, we create branches corresponding to the third outcome of \cref{lem:magic-lemma}. For each vertex $v \in N_{L_3}^{(3)}(r)$, we create a branch for each possible valid combination of a 3-colouring of $N_{L_3}[v]$, a 3-colouring of $\{ n_r(a)\, : \, a \in N_{L_3}(v)\}$, a set $S \subseteq N_{L_3}^{(2)}(v)$ of size at most $3\mu^{2/3 - \varepsilon}$ such that $\{x \in N_{L_3}^{(2)}(v) \, : \, p_v(x) \not \in L_3\setminus N^{(3)}(r)\} \subseteq S$, and each 3-colouring of $S$. In these branches, colour the vertices of \[
            N_{L_3}[v]\cup \{ n_r(a)\, : \, a \in N_{L_3}(v)\}\cup S
            \]
            according to the colourings defining the branch. This gives a partial colouring $\phi$. We extend $\phi$ to the remaining vertices $x \in N_{L_3}^{(2)}(v) \setminus S$ via $\phi(x):= \phi(n_r(p_v(x)))$.
        \end{itemize}
        \begin{claim}
            If branching rule B5 was applied, the number of instances is at most
            \[6\mu^{4\mu^{2/3-\varepsilon}} \cdot F\left(\frac{6\mu}{7}\right) + \mu^{8\mu^{2/3-\varepsilon}}\cdot F\left(4\mu^{2/3 + 2\varepsilon} \right) + p(n).\]
        \end{claim}
        \begin{subproof}
            We first consider the branches corresponding to the second outcome of \cref{lem:magic-lemma}. 
            For each valid choice of $T$ and $\widetilde{T}$, both of which need to be subsets of $L_3$ of size at most $2\mu^{2/3-\varepsilon}$, we create $6$ new instances, so the number of such branches is bounded by
            \[
           6 \cdot (\text{\# of choices for $T$}) \cdot (\text{\# of choices for $\widetilde{T}$}) \leq 6 \cdot \left(|L_3|^{2\mu^{2/3-\varepsilon}}\right)^2 = 6 \mu^{4\mu^{2/3-\epsilon }}.
            \]      
            Next, we consider the branches corresponding to the third outcome of \cref{lem:magic-lemma}. There are at most $\mu$ choices for $v \in N_{L_3}^{(3)}(r)$. Since the instance is reduced, $|N_{L_3}(v)| < \mu^{1/3 + \varepsilon}$ and so there are at most
        $3^{2\mu^{1/3+\varepsilon}}$ options for 3-colourings of $N_{L_3}[v]\cup \{n_r(a):a\in N_{L_3}(v)\}$. Since $S$ is chosen from $L_3$ of size at most $3\mu^{2/3-\varepsilon}$, there are at most $\mu^{3\mu^{2/3-\varepsilon}}$ choices for $S$. For each choice of $S$, there are at most $3^{3\mu^{2/3-\varepsilon}}$ possible 3-colourings of $S$. The number of instances that we create is hence at most
      \[
      3^{2\mu^{1/3+\varepsilon}} \cdot\mu^{3\mu^{2/3-\varepsilon}} \cdot 3^{3\mu^{2/3-\varepsilon}}\leq \mu^{8\mu^{2/3-\varepsilon}}.
      \]Here we used that $\varepsilon<\frac16$ and $\mu\geq 3$.
            The desired recurrence relation now follows from \cref{lem:size-of-L3-after-applying-B5}.
        \end{subproof}

     We prove by induction on $\mu$ that $F(\mu)\leq C\exp(\mu^{2/3-\varepsilon}(\log \mu)^2)$ for some constant $C=C(\varepsilon)>0$. By choosing the constant $C$ above and the constant $c$ in R4 appropriately, we may assume the claim holds for all $\mu'<\mu$ where $\mu$ is `sufficiently large'. 
     In particular, as $\varepsilon<1/6$, we may assume that $4\mu^{2/3+2\varepsilon}\leq (6/7)\mu$ and therefore when B5 is applied, the claim above and the induction hypothesis imply that the running time is at most
     \[
     7\exp\left(8\mu^{2/3-\varepsilon}\log\mu\right) \cdot F(6\mu/7) \leq 7C\exp\left(8\mu^{2/3-\varepsilon}\log \mu + (6\mu/7)^{2/3 -\varepsilon}(\log (6\mu/7))^2\right),
     \]
     which for $\mu$ sufficiently large is less than $C\exp(\mu^{2/3 - \varepsilon}(\log \mu)^2)$, giving the desired running time in the case that B5 is applied.
     For each of the branching rules B1-B4, using the claims above and the induction hypothesis, the running time may be bounded by
     \begin{align*}
         F(\mu - 1) &+ 6F(\mu - \mu^{1/3 + \varepsilon}/3)\\
         &\leq F(\mu)^{(1-1/\mu)^{2/3-\varepsilon}}+6F(\mu)^{(1-\mu^{(1/3+\epsilon)-1}/3)^{2/3-\varepsilon}}\\
         &\leq  F(\mu)^{1-1/(3\mu)}+6F(\mu)^{1-1/(9\mu^{2/3-\epsilon})}\\
         &\leq F(\mu)(\exp(-1/(3\mu)\cdot \mu^{2/3-\varepsilon}(\log \mu)^2)+6\exp(-1/(9\mu^{2/3-\epsilon})\cdot \mu^{2/3-\varepsilon}(\log \mu))\\
         &= F(\mu)(\exp(-1/3 \cdot \mu^{-1/3+\varepsilon}(\log \mu)^2)+6\exp(-1/9 \cdot (\log \mu)^2))\\
         &\leq F(\mu).
     \end{align*}
    In the second inequality, we used that for $x$ sufficiently small and $\varepsilon<\frac13$, $(1-x)^{2/3-\varepsilon}\leq 1-(2/3-\varepsilon)x\leq 1-x/3$. For the last step we use that for $\mu$ sufficiently large, $\exp(-1/3 \cdot \mu^{-1/3+\varepsilon}(\log \mu)^2)+6\exp(-1/9 \cdot (\log \mu)^2) < 1$. Namely, while the term $\exp(-1/3 \cdot \mu^{-1/3+\varepsilon}(\log \mu)^2)$ asymptotically approaches $1$ from below, the term $6\exp(-1/9 \cdot (\log \mu)^2)$ tends to $0$ asymptotically significantly faster, allowing the asymptotic behaviour to be dominated by the first term.

    Hence, we have proved that $F(\mu)\leq C\exp(\mu^{2/3-\varepsilon}(\log \mu)^2)$. During the reductions towards each instance, all steps are performed in time polynomial in $n$. Since $\mu\leq n$, for any $\varepsilon'>0$ there is a constant $C=C(\varepsilon,\varepsilon')$ such that the total running time will be at most $2^{Cn^{2/3-\varepsilon+\varepsilon'}}$, absorbing the additional polynomial factor. Since this holds for any $0<\varepsilon<1/33$ and $\varepsilon'>0$, we may obtain the running time of $2^{O(n^{2/3-\varepsilon})}$ for any $\varepsilon<1/33$, as desired.
    \end{proof}

\end{document}